\documentclass[11pt]{amsart}

\usepackage{amsmath}
\usepackage{amsfonts}
\usepackage{amssymb}
\usepackage{amscd}
\usepackage{amsthm}
\usepackage{framed}
\usepackage{fullpage}
\usepackage{graphicx}
\usepackage{latexsym}
\usepackage[numbers]{natbib}  

\usepackage{hyperref}
\hypersetup{colorlinks,linkcolor={red},citecolor={blue},urlcolor={blue}}

\newtheorem{theorem}{Theorem}[section]

\theoremstyle{definition}
\newtheorem{definition}[theorem]{Definition}

\newtheorem{remark}[theorem]{Remark}

\numberwithin{equation}{section}

\newcommand{\CC}{\mathbb C}
\newcommand{\HH}{\mathbb H}
\newcommand{\QQ}{\mathbb Q}
\newcommand{\RR}{\mathbb R}
\newcommand{\ZZ}{\mathbb Z}
\newcommand{\SL}{\mathop{\mathrm {SL}}\nolimits}
\newcommand{\Orth}{\mathop{\null\mathrm {O}}\nolimits}
\newcommand{\rank}{\mathop{\mathrm {rk}}\nolimits}
\newcommand{\latt}[1]{{\langle{#1}\rangle}}

\newenvironment{psmallmatrix}
  {\left(\begin{smallmatrix}}
{\end{smallmatrix}\right)}

\begin{document}

\title[There are no extremal eutactic stars other than root systems]{There are no extremal eutactic stars other than root systems}

\author{Haowu Wang}

\address{School of Mathematics and Statistics, Wuhan University, Wuhan 430072, Hubei, China}

\email{haowu.wangmath@gmail.com}


\date{\today}


\begin{abstract}
A eutactic star on an integral lattice is called extremal if it induces a holomorphic Jacobi form of lattice index and singular weight via the theta block. The famous Macdonald identities imply that root systems are extremal as eutactic stars. In this paper we prove that every extremal eutactic star arises as a root system. This answers a question posed by Skoruppa.
\end{abstract}

\maketitle

\section{Introduction and the statement of the main result}
Let $L$ be an integral positive definite lattice with the bilinear form $(-,-)$ and dual lattice $L'$. A finite family $\mathbf{s}$ of nonzero elements $s_j$ in $L'$  ($1\leq j\leq N$) is called a \textit{eutactic star} on $L$ if it satisfies
\begin{equation*}
\sum_{j=1}^N (s_j, x)^2 = (x,x), \quad \text{for all $x\in L$}.      
\end{equation*}
Equivalently, the family $\mathbf{s}$ induces an isometric embedding 
$$
\iota_\mathbf{s} : L \to \ZZ^N, \quad x \mapsto \big((s_j,x) : 1\leq j\leq N \big). 
$$
Vice versa, any isometric embedding from $L$ to $\ZZ^N$ may be realized in this way. 

A eutactic star on $L$ also induces holomorphic Jacobi forms of lattice index $L$ via theta blocks; see the function $\vartheta_\mathbf{s}(\tau, \mathfrak{z})$ defined in \eqref{eq:theta-block} below. Let $v_\eta$ denote the multiplier system 
of the Dedekind eta function
$$
\eta(\tau) = q^{\frac{1}{24}}\prod_{n=1}^\infty(1-q^n), \quad \tau \in \HH, \; q=e^{2\pi i \tau} 
$$
as a modular form of weight $1/2$ on $\SL_2(\ZZ)$. We define the shadow of $L$ as
$$
L^\bullet = \{ x \in L\otimes\QQ : (x,y) - (y,y)/2 \in \ZZ \quad \text{for all $y\in L$} \}.
$$
Note that $L^\bullet = L'$ if $L$ is an even lattice. Following \cite{Gri88, GSZ19} one defines Jacobi forms of lattice index, which are a generalization of classical Jacobi forms introduced by Eichler and Zagier \cite{EZ85}. 
\begin{definition}
Let $k$ be integral or half-integral and $D$ be an integer modulo $24$. A holomorphic function $\varphi(\tau,\mathfrak{z}) : \HH \times (L\otimes \CC) \rightarrow \CC$ is called a \textit{holomorphic Jacobi form} of weight $k$, character $v_\eta^D$ and index $L$, if it satisfies 
    \begin{align*}
    \varphi \left( \frac{a\tau +b}{c\tau + d},\frac{\mathfrak{z}}{c\tau + d} 
    \right)& = v_\eta(A)^D(c\tau + d)^k 
    \exp{\left(\pi i\frac{c(\mathfrak{z},\mathfrak{z})}{c 
            \tau + d}\right)} \varphi ( \tau, \mathfrak{z} ),  \\
    \varphi (\tau, \mathfrak{z}+ x \tau + y)&=(-1)^{(x,x)+(y,y)}\exp\left( -\pi i\big( (x,x)\tau + 2(x,\mathfrak{z}) \big) \right) 
    \varphi (\tau, \mathfrak{z} ),       
    \end{align*}
    for all $A=\begin{psmallmatrix} a & b \\ c & d \end{psmallmatrix} \in \SL_2(\ZZ)$ 
    and $x,y \in L$, and if its Fourier expansion takes the form
    \begin{equation}\label{eq:Fourier}
    \varphi ( \tau, \mathfrak{z} )= \sum_{\substack{n\in \frac{D}{24}+\ZZ, \; \ell \in L^\bullet \\ 2n\geq (\ell,\ell) }}f(n,
    \ell)q^n\zeta^\ell, \quad \zeta^\ell=e^{2\pi i (\ell,\mathfrak{z})}.
    \end{equation}  
\end{definition}
From the theta decomposition of Jacobi forms we conclude that $k\geq \frac{1}{2}\rank(L)$ if $\varphi$ is non-constant, where $\rank(L)$ denotes the rank of $L$. The smallest possible weight of a non-constant holomorphic Jacobi form of index $L$, i.e. $\frac{1}{2}\rank(L)$ is called the \textit{singular} weight. 

The Jacobi triple product formula
$$
\vartheta(\tau,z) = q^{\frac{1}{8}}(\zeta^{\frac{1}{2}}-\zeta^{-\frac{1}{2}}) \prod_{n=1}^\infty (1-q^n\zeta)(1-q^n\zeta^{-1})(1-q^n), \quad z\in \CC, \; \zeta=e^{2\pi i z}
$$
defines a holomorphic Jacobi form of singular weight $\frac{1}{2}$, character $v_\eta^3$ and index $\ZZ$ (see e.g. \cite{GN98, GSZ19}). Let $\mathbf{s}=(s_j: 1\leq j \leq N)$ be a eutactic star on $L$ and $l$ be a positive integer. Gritsenko, Skoruppa and Zagier \cite{GSZ19} defined a holomorphic Jacobi form of weight $N/2$ and index $L$ as 
\begin{equation}\label{eq:theta-block}
\vartheta_\mathbf{s}(\tau, \mathfrak{z}) := \prod_{j=1}^N \vartheta(\tau, (s_j, \mathfrak{z})),
\end{equation}
and they further considered the function $\eta(\tau)^{l-N} \vartheta_\mathbf{s}(\tau, \mathfrak{z})$. Such functions are called \textit{theta blocks in several variables} following \cite{GSZ19}. 
In general, $\eta^{l-N} \vartheta_\mathbf{s}$ is no longer a holomorphic Jacobi form, because it may not be holomorphic at infinity, i.e. the condition $2n\geq (\ell,\ell)$ in Fourier expansion \eqref{eq:Fourier} may not hold. The smallest possible $l$ such that $\eta^{l-N} \vartheta_\mathbf{s}$ defines a holomorphic Jacobi form is $\rank(L)$. 

A eutactic star $\mathbf{s}$ on $L$ is called \textit{extremal} if the associated function 
$$
\vartheta_\mathbf{s}^*(\tau, \mathfrak{z}) :=\eta(\tau)^{\rank(L)-N} \vartheta_\mathbf{s}(\tau, \mathfrak{z})
$$
is a holomorphic Jacobi form of singular weight and index $L$. A generalization of \cite[Proposition 5.2]{GSZ19} yields that $\mathbf{s}$ is extremal if and only if the inequality
$$
\min_{x\in L\otimes\RR} \sum_{j=1}^N B\big( (s_j,x) \big) \geq \frac{N-\rank(L)}{24}
$$
holds, where 
$$
B(x)=\frac{1}{2}\left( y - \frac{1}{2} \right)^2, \quad y-x\in \ZZ, \; 0\leq y < 1.
$$

It is a particularly interesting and highly non-trivial question to look for extremal eutactic stars. All known examples are related to root systems. Let $R$ be an irreducible root system with the normalized bilinear form $\latt{-,-}$ such that $\latt{r,r}=2$ for long roots $r$. We denote the dual Coxeter number and a set of positive roots of $R$ by $h$ and $R^+$, respectively. Then we have
$$
\sum_{r\in R^+} \latt{r, z}^2 = h\latt{z,z}, \quad z\in R\otimes\CC.
$$
Let $P$ denote the integral lattice 
$$
\{ x \in R\otimes\QQ : \latt{x,r} \in \ZZ, \quad \text{for all $r\in R$} \} 
$$
equipped with the bilinear form
$$
(x,x):=h\latt{x,x}, \quad x\in P. 
$$
We then have the isometric embedding
$$
P \to \ZZ^{|R^+|}, \quad x \mapsto \big((r/h,x)=\latt{r,x} : r\in R^+\big). 
$$
Thus the family $\mathbf{s}_R:=(r/h: r\in R^+)$ defines a eutactic star on $P$. As observed by Gritsenko, Skoruppa and Zagier \cite{GSZ19}, the Macdonald identity \cite{Mac72} implies that $\mathbf{s}_R$ is extremal, and the associated function $\vartheta_{\mathbf{s}_R}^*$ coincides with the product side of the denominator identity of the affine Lie algebra of type $R$ (see \cite{KP84}). 

At both conferences in Darmstadt in 2019 and Sochi in 2020, Skoruppa asked whether there are extremal eutactic stars other than root systems. In this paper, we give a negative answer to Skoruppa's question.

\begin{theorem}\label{MTH}
Let $\mathbf{s}$ be an extremal eutactic star on $L$. Then the set $\{ x \in L': x \in \mathbf{s} \; \text{or} \; -x \in \mathbf{s}\}$ is isomorphic to a root system of the same rank as $L$. 
\end{theorem}

The proof of the theorem is inspired by our previous joint work \cite{WW23} with Brandon Williams, in which we use the Laplace operator on a tube domain to show the non-existence of holomorphic Borcherds products of singular weight on $\Orth(l,2)$ with $l> 26$. In the next section, we employ the heat operator on Jacobi forms to prove Theorem \ref{MTH}. 

\section{A proof of Theorem \ref{MTH}}
Let $\{\alpha_1,...,\alpha_l\}$ be a basis of $L\otimes\RR$ and $\{\alpha_1^*,...,\alpha_l^*\}$ be the dual basis. We write 
$$
\mathfrak{z}=\sum_{i=1}^l z_i\alpha_i \in L\otimes \CC \quad \text{and} \quad \frac{\partial}{\partial\mathfrak{z}}=\sum_{i=1}^l\alpha_i^* \frac{\partial}{\partial z_i}, \quad z_i\in \CC.
$$
The heat operator is defined as
$$
H=\frac{1}{2\pi i}\frac{\partial}{\partial\tau}+\frac{1}{8\pi^2}\left(\frac{\partial}{\partial\mathfrak{z}},\frac{\partial}{\partial\mathfrak{z}}\right).
$$
It is clear that $H$ is independent of the choice of basis. 
This type of operator was first used in \cite{EZ85} to construct differential operators on classical Jacobi forms, and later generalized to Jacobi forms of lattice index in \cite{CK00}.  

Let $\varphi$ be a holomorphic Jacobi form of singular weight and index $L$. By the theta decomposition, $\varphi$ is a $\CC$-linear combination of Jacobi theta functions of $L$ (see e.g. \cite[Section 4]{Gri94} or \cite[Section 12]{GSZ19}). The operator $H$ acts on the Fourier expansion of $\varphi$ via
$$
H(q^n \zeta^\ell) = \left( n-\frac{1}{2}(\ell,\ell) \right) q^n\zeta^\ell. 
$$
Therefore, $H(\varphi)$ is identically zero. Conversely, if a non-constant holomorphic Jacobi form $\phi$ satisfies $H(\phi)=0$, then it is of singular weight (see \cite[Lemma 4.1]{Gri94}).  

We first describe zeros of holomorphic Jacobi forms of singular weight. 

\begin{theorem}\label{th}
Let $\varphi$ be a non-constant holomorphic Jacobi form of singular weight and index $L$. Let $v$ be a nonzero vector of $L'$. Assume that $\varphi$ vanishes on the set
$$
v^\perp:= \{ (\tau,\mathfrak{z})\in \HH \times(L\otimes\CC) : (v,\mathfrak{z})=0 \}.
$$
Then $v^\perp$ has multiplicity one in the divisor of $\varphi$ and the identity 
$$
\varphi(\tau,\sigma_v(\mathfrak{z})) = - \varphi(\tau, \mathfrak{z})
$$
holds for any $(\tau,\mathfrak{z})\in \HH \times(L\otimes\CC)$, where $\sigma_v$ is the reflection fixing $v^\perp$ defined as
$$
\sigma_v(x) = x - \frac{2(v,x)}{(v,v)}v, \quad x\in L.
$$
\end{theorem}

\begin{proof}
Let $L_v$ denote the orthogonal complement of $v$ in $L$. We write $\mathfrak{z}=zv+z'$ for $z\in \CC$ and $z'\in L_v\otimes\CC$. Let $d$ be the multiplicity of $v^\perp$ in the divisor of $\varphi$, that is, the Taylor expansion of $\varphi$ at $z=0$ takes the form
$$
\varphi(\tau,\mathfrak{z}) = f_d(\tau,z')z^d + O(z^{d+1}), \quad f_d(\tau,z')\not\equiv 0.
$$
By assumption, $d\geq 1$. 
For the basis of $L\otimes\RR$, we fix $\alpha_1=v$ and $\alpha_2$, ..., $\alpha_l$ to be a basis of $L_v$. By applying the corresponding heat operator to $\varphi$, we derive
$$
H(\varphi) = \varepsilon d(d-1) f_d(\tau,z')z^{d-2} + O(z^{d-1}),
$$
where $\varepsilon$ is a nonzero constant. Since $\varphi$ is of singular weight, $H(\varphi)=0$, as mentioned at the beginning of this section. As the leading term, $\varepsilon d(d-1) f_d(\tau,z')z^{d-2}$ has to be zero, which yields that $d=1$. 
To prove the last claim, we introduce the function
$$
\phi(\tau,\mathfrak{z}):=\varphi(\tau,\sigma_v(\mathfrak{z})) + \varphi(\tau,\mathfrak{z}).
$$
The Taylor expansion of $\phi$ at $z=0$ starts with
$$
\phi(\tau,\mathfrak{z})= \varphi(\tau, -zv + z') + \varphi(\tau,zv+z') = O(z^2).
$$
Obviously, $\phi$ also vanishes on $v^\perp$ and $H(\phi)=0$. If $\phi$ is not identically zero, then by an argument similar to the above, we prove that $\phi$ vanishes on $v^\perp$ with multiplicity one, which contradicts the Taylor expansion of $\phi$ above. Therefore, $\phi=0$. The proof is complete. 
\end{proof}

\begin{remark}
From the proof above we can see that Theorem \ref{th} holds even for any non-constant holomorphic function $\varphi$ on $\HH \times (L\otimes\CC)$ that satisfies $H(\varphi)=0$. 
\end{remark}

We now prove Theorem \ref{MTH}.

\begin{proof}[Proof of Theorem \ref{MTH}]
By assumption, the function 
$$
\vartheta_\mathbf{s}^*(\tau,\mathfrak{z}) = \eta(\tau)^{\rank(L)} \prod_{j=1}^N \frac{\vartheta(\tau,(s_j,\mathfrak{z}))}{\eta(\tau)}
$$
is a holomorphic Jacobi form of singular weight and index $L$. It is well known that $\vartheta(\tau,z)$ vanishes precisely with multiplicity one on the set $\{ (\tau,z)\in \HH\times\CC : z\in \ZZ\tau +\ZZ \}$. Therefore, $\vartheta_\mathbf{s}^*(\tau,\mathfrak{z})=0$ if and only if there exists $1\leq j \leq N$ such that $(s_j,\mathfrak{z})\in \ZZ\tau + \ZZ$. We need to show that the family
$$
\mathcal{S}:=\big( x \in L': x \in \mathbf{s} \; \text{or} \; -x \in \mathbf{s}\big)
$$
defines a root system. 

Let $x, y \in \mathcal{S}$. We claim that there is no integer $m>1$ such that $mx \in \mathcal{S}$, otherwise the multiplicity of $x^\perp$ in the divisor of $\vartheta_\mathbf{s}^*$ would be not simple, a contradiction by Theorem \ref{th}. A similar argument shows that the elements of the family $\mathcal{S}$ are mutually distinct. By Theorem \ref{th}, we have
$$
\vartheta_\mathbf{s}^*(\tau,\sigma_x(\mathfrak{z})) = -\vartheta_\mathbf{s}^*(\tau,\mathfrak{z}).
$$
Therefore, $\vartheta_\mathbf{s}^*(\tau,\mathfrak{z})=0$ if 
$$
(\mathfrak{z},\sigma_x(y)) = (\sigma_x(\mathfrak{z}), y) \in \ZZ\tau + \ZZ.
$$
The shape of the divisor of $\vartheta_\mathbf{s}^*$ implies that $\sigma_x(y) \in \mathcal{S}$. 

It remains to show that $2(x,y)/(x,x)\in \ZZ$. To do it, we have to study the divisor of type $(x,\mathfrak{z})=\tau$. We use the Laplace operator to prove it as the proof of \cite[Theorem 2.1]{WW23}. 

Let $U$ be the unique even unimodular lattice of signature $(1,1)$ and $M=U\oplus L$. Let $\{\beta_1,...,\beta_{l+2}\}$ be a basis of $M\otimes\RR$ and $\{\beta_1^*,...,\beta_{l+2}^* \}$ be the dual basis. We define the Laplace operator as 
$$
\mathbf{\Delta} = \left(\frac{\partial}{\partial Z},\frac{\partial}{\partial Z}\right), \quad Z=\sum_{j=1}^{l+2} z_j\beta_j \in M\otimes \CC, \quad \frac{\partial}{\partial Z}=\sum_{j=1}^{l+2}\beta_j^* \frac{\partial}{\partial z_j}, \quad z_j\in \CC.
$$
It is clear that $\mathbf{\Delta}$ is independent of the choice of basis. For any $\lambda \in M\otimes\QQ$, we have 
$$
\mathbf{\Delta} e^{2\pi i (\lambda, Z)} = -4\pi^2 (\lambda,\lambda) e^{2\pi i (\lambda, Z)}.
$$ 

Write a vector $\lambda\in M'=U\oplus L'$ as $(n,v,m)$ for $n, m\in \ZZ$ and $v\in L'$ with $(\lambda,\lambda)=(v,v)-2nm$. We introduce an auxiliary variable $w\in \HH$ and define a holomorphic function as
$$
\widehat{\vartheta}_\mathbf{s}^*(Z):=\vartheta_\mathbf{s}^*(\tau,\mathfrak{z})e^{2\pi iw}, \quad Z=(\tau,\mathfrak{z},w)\in \HH\times(L\otimes\CC)\times\HH \subsetneq M\otimes\CC. 
$$
The Fourier series of $\widehat{\vartheta}_\mathbf{s}^*(Z)$ are supported only on norm-zero vectors of $M\otimes\QQ$. 
Clearly, 
$$
\mathbf{\Delta} = -2 \frac{\partial}{\partial \tau} \frac{\partial}{\partial w} + \left(\frac{\partial}{\partial\mathfrak{z}},\frac{\partial}{\partial\mathfrak{z}}\right) \quad \text{and} \quad \mathbf{\Delta}\left(\widehat{\vartheta}_\mathbf{s}^*\right)=0.
$$
For any $x\in \mathcal{S}$ we define $\lambda_x:=(0,x,1)\in M'$. Let $K$ denote the orthogonal complement of $\lambda_x$ in $M$.
We write $Z=z\lambda_x+Z'$ for $z\in \CC$ and $Z' \in K\otimes\CC$, and expand $\widehat{\vartheta}_\mathbf{s}^*(Z)$ into Taylor series at $z=0$ as
$$
\widehat{\vartheta}_\mathbf{s}^*(Z) = F_d(Z')z^d+O(z^{d+1}), \quad F_d(Z')\not\equiv 0.
$$
Here $d\geq 1$, because $\vartheta_\mathbf{s}^*(\tau,\mathfrak{z})=0$ whenever $(x,\mathfrak{z})\in \ZZ\tau + \ZZ$. From $\mathbf{\Delta}\big(\widehat{\vartheta}_\mathbf{s}^*\big)=0$ we further deduce that $d=1$, that is, $\widehat{\vartheta}_\mathbf{s}^*(Z)$ vanishes with multiplicity one on the quadratic divisor
$$
\lambda_x^\perp=\{ Z\in \HH\times(L\otimes\CC)\times\HH : (\lambda_x, Z)=0, \; \text{i.e.} \; (x,\mathfrak{z})=\tau \}.
$$
Therefore, we have the Taylor expansion
$$
\widehat{\vartheta}_\mathbf{s}^*(Z) = F_1(Z') z + O(z^2).
$$
Recall that the reflection fixing $\lambda_x^\perp$ is defined as
$$
\sigma_{\lambda_x}(\mu) = \mu - \frac{2(\lambda_x, \mu)}{(\lambda_x,\lambda_x)}\lambda_x, \quad \mu \in M.
$$
We apply the Laplace operator $\mathbf{\Delta}$ to the function 
$$
\hat{\phi}(Z):=\widehat{\vartheta}_\mathbf{s}^*(\sigma_{\lambda_x}(Z))+ \widehat{\vartheta}_\mathbf{s}^*(Z) = \widehat{\vartheta}_\mathbf{s}^*(-z\lambda_x+Z')+\widehat{\vartheta}_\mathbf{s}^*(z\lambda_x+Z') = O(z^2)
$$
and find that $\mathbf{\Delta}(\hat{\phi})=0$, which forces that $\widehat{\vartheta}_\mathbf{s}^*(\sigma_{\lambda_x}(Z))=-\widehat{\vartheta}_\mathbf{s}^*(Z)$ as in the previous proof of Theorem \ref{th}. Let $y\in \mathcal{S}$ and $\lambda_y=(0,y,1)\in M'$. Note that $\widehat{\vartheta}_\mathbf{s}^*(Z)$ also vanishes on $\lambda_y^\perp$.  Therefore, $\widehat{\vartheta}_\mathbf{s}^*(Z)$ vanishes on the quadratic divisor orthogonal to the vector
 $$
\sigma_{\lambda_x}(\lambda_y) = \lambda_y - \frac{2(\lambda_x,\lambda_y)}{(\lambda_x,\lambda_x)}\lambda_x = \left( 0, \sigma_x(y), 1-\frac{2(x,y)}{(x,x)}\right).
 $$
It follows that $\vartheta_\mathbf{s}^*(\tau,\mathfrak{z})=0$ if 
$$
\left(\sigma_x(y),\mathfrak{z}\right) = \left(1 - \frac{2(x,y)}{(x,x)}\right)\tau.
$$
From the shape of the zeros of $\vartheta_\mathbf{s}^*$ described above, we conclude that $2(x,y)/(x,x)$ is integral. 
\end{proof}

\bigskip

\noindent
\textbf{Acknowledgements} 
The author thanks Nils Skoruppa for valuable discussions and for helpful
comments on an earlier version of this paper. The author also thanks the two referees for their useful suggestions.

\bibliographystyle{plainnat}
\bibliofont
\bibliography{refs}

\begin{thebibliography}{9}
\providecommand{\natexlab}[1]{#1}
\providecommand{\url}[1]{\texttt{#1}}
\expandafter\ifx\csname urlstyle\endcsname\relax
  \providecommand{\doi}[1]{doi: #1}\else
  \providecommand{\doi}{doi: \begingroup \urlstyle{rm}\Url}\fi

\bibitem[Choie and Kim(2000)]{CK00}
YoungJu Choie and Haesuk Kim.
\newblock Differential operators on {J}acobi forms of several variables.
\newblock \emph{J. Number Theory}, 82\penalty0 (1):\penalty0 140--163, 2000.

\bibitem[Eichler and Zagier(1985)]{EZ85}
Martin Eichler and Don Zagier.
\newblock \emph{The theory of {J}acobi forms}, volume~55 of \emph{Progress in
  Mathematics}.
\newblock Birkh\"{a}user Boston, Inc., Boston, MA, 1985.

\bibitem[Gritsenko(1988)]{Gri88}
V.~A. Gritsenko.
\newblock Fourier-{J}acobi functions in {$n$} variables.
\newblock \emph{Zap. Nauchn. Sem. Leningrad. Otdel. Mat. Inst. Steklov.
  (LOMI)}, 168\penalty0 (Anal. Teor. Chisel i Teor. Funktsi\u{\i}. 9):\penalty0
  32--44, 187--188, 1988.

\bibitem[Gritsenko and Nikulin(1998)]{GN98}
Valeri~A. Gritsenko and Viacheslav~V. Nikulin.
\newblock Automorphic forms and {L}orentzian {K}ac-{M}oody algebras. {II}.
\newblock \emph{Internat. J. Math.}, 9\penalty0 (2):\penalty0 201--275, 1998.

\bibitem[Gritsenko(1994)]{Gri94}
Valery Gritsenko.
\newblock Modular forms and moduli spaces of abelian and {$K3$} surfaces.
\newblock \emph{Algebra i Analiz}, 6\penalty0 (6):\penalty0 65--102, 1994.

\bibitem[Gritsenko et~al.(2024)Gritsenko, Skoruppa, and Zagier]{GSZ19}
Valery Gritsenko, Nils-Peter Skoruppa, and Don Zagier.
\newblock Theta blocks.
\newblock J. Eur. Math. Soc., publish online, 2024.
\newblock URL \url{https://doi.org/10.4171/jems/1471}.

\bibitem[Kac and Peterson(1984)]{KP84}
Victor~G. Kac and Dale~H. Peterson.
\newblock Infinite-dimensional {L}ie algebras, theta functions and modular
  forms.
\newblock \emph{Adv. in Math.}, 53\penalty0 (2):\penalty0 125--264, 1984.

\bibitem[Macdonald(1972)]{Mac72}
I.~G. Macdonald.
\newblock Affine root systems and {D}edekind's {$\eta $}-function.
\newblock \emph{Invent. Math.}, 15:\penalty0 91--143, 1972.

\bibitem[Wang and Brandon(2023)]{WW23}
Haowu Wang and Williams Brandon.
\newblock On the non-existence of singular {B}orcherds products.
\newblock preprint, 2023.
\newblock URL \url{arXiv:2301.13367}.

\end{thebibliography}

\end{document}